\documentclass[a4paper,reqno]{amsart}
\usepackage{amssymb, amsmath, amscd}
\usepackage[final]{graphicx}
\usepackage{float}\usepackage{wrapfig}
\usepackage{color}
\usepackage{bbm}
\usepackage[final]{graphicx}
\usepackage{tikz-cd}
\usepackage{color}

\definecolor{brown(traditional)}{rgb}{0.59, 0.29, 0.0}

\def\Tr{\operatorname{Tr}}
\def\Gronk{[0,\iota_P)}
\def\rr{{\mathfrak{r}}}
\newtheorem{theorem}{Theorem}[section]

\newtheorem{lemma}[theorem]{Lemma}

\makeatletter
 \@addtoreset{equation}{section}
\makeatother
\begin{document}
\title[Harmonic radial vector fields on harmonic spaces]{Harmonic radial vector fields\\ on harmonic spaces}
\author{P. B. Gilkey and J. H. Park}
\address{PG: Mathematics Department, University of Oregon, Eugene OR 97403-1222, USA}
\email{gilkey@uoregon.edu}
\address{JHP: Department of
Mathematics, Sungkyunkwan University, Suwon, 16419 Korea.}
\email{parkj@skku.edu}
\begin{abstract} We characterize harmonic spaces in terms of the dimensions of various spaces of radial
eigen-spaces of the Laplacian $\Delta^0$ on functions and the Laplacian $\Delta^1$ on 1-forms. We
examine the nature of the singularity as the geodesic
distance $r$ tends to zero of radial eigen-functions and 1-forms.
Via duality, our results give rise to corresponding results
for radial vector fields. Many of our results extend to the context of spaces which are harmonic with
respect to a single point.
\end{abstract}
\maketitle
\section{Introduction}

In Euclidean space $\mathbb{R}^m$,
there is a ``radial" solution to the Laplace equation $\Delta f = 0$ which is given by
setting $ f(x) = \|x\|^{2-m}$ for $m > 2$, and
$f(x)=\log \|x\|^2$ for $m=2$ where $\|x\|^2=(x^1)^2+\dots+(x^m)^2$
and where $\Delta=-\partial_{x^1}^2-\dots-\partial_{x^m}^2$.
Ruse \cite{R31} was the first to study this problem for an arbitrary
Riemannian manifold $\mathcal{M}=(M,g)$. Fix a point $P$ of $M$ and let $r_P(Q)$ be the
geodesic distance from a point $Q$ of $M$ to $P$. Let $\Delta=\delta d$ be the Laplace-Beltrami
operator. A function is said to be \emph{radial} if $f(Q)=f(r_P(Q))$. Let $\iota_P$ be the injectivity
radius. If there exists a radial function so that
$\Delta f=0$ for $0<r<\iota_P$, then $\mathcal{M}$ is said to be \emph{harmonic with respect to $P$}. If $\mathcal{M}$ is
harmonic with respect to every point, then $\mathcal{M}$ is said to be
a \emph{harmonic space} (see Willmore~\cite{W}). Let $\phi$ be a smooth positive function with $\phi(0)=0$.
The Riemannian metric $g_\phi:=e^{2\phi(\|x\|^2)}((dx^1)^2+\dots+(dx^m)^2)$ is harmonic with respect to the
origin but, for generic $\phi$, is not harmonic with respect to any other point, see Copson and Ruse~\cite{CR40}.
We will present other examples in Section~\ref{S1.7}.

In geodesic normal coordinates centered at $P$, we can express the Riemannian measure
$\operatorname{dvol}=\tilde\Theta_Pdx^1\dots dx^m$ where
$\tilde\Theta_P=\sqrt{\det(g_{ij})}$ is called the
\emph{volume density function} in geodesic normal coordinates; $\mathcal{M}$ is a
harmonic space if and only if $\tilde\Theta_P$ is radial for every point $P$ of $M$. In this setting,
the volume density
function is independent of the particular point chosen. The study of harmonic spaces is an active and continuing
area of research and there are many results available.
For example,
Heber \cite{H} proved that a simply connected homogeneous harmonic space is either flat, or a
rank one symmetric space, or a Damek Ricci space.
Nikolayevsky \cite{N} proved that the density function of a non-compact harmonic space
is an exponential polynomial.
Ramachandran and Ranjan~\cite{RR} classified certain non-compact
harmonic spaces in terms of their density functions.
Itoh et al.~\cite{IP} studied harmonic spaces in relation to prescribed Ricci curvature and volume entropy.
Choe, Kim and Park~\cite{CHP} characterized certain harmonic spaces in terms of
the radial eigen-functions of the Laplacian.
Gilkey and Park~\cite{GP} classified harmonic space by using the asymptotic series
of the density function and the eigenvalues of the Jacobi operator.
We also refer to earlier work by Ruse~\cite{R63}, Ruse, Walker,
and Willmore~\cite{RWW61}, and Thomas and Titt~\cite{TT}.

We shall primarily be interested in analytic properties of harmonic spaces. There are,
however, many other characterizations of harmonic spaces. For example, it is known that
$\mathcal{M}$ is a harmonic
space if and only if every sufficiently small geodesic sphere has constant mean curvature.
 Szab\'o~\cite{S} proved that in a harmonic space, the volume of the intersection of
 two geodesic balls of small radii
depends only on the radii and the distance between the centers.  Csik\'os and Horv\'ath~\cite{CH1, CH2,CH3}
gave several characterizations of harmonic spaces in terms of volumes of geodesic balls or volumes of tubes.
Give the sphere bundle $S(M)$ of unit tangent vector fields the Sasaki metric.
If $\xi$ is a unit tangent vector field, then $\xi$ defines a smooth map from $M$ to $S(M)$.
 Boeckx and Vanhecke \cite{BV} studied the energy functional of these maps and
 showed that $\mathcal{M}$  is a harmonic space if and only if each radial unit vector field is a critical point of
 the energy functional within the class of unit tangent vector fields.

Flat space, the rank one symmetric spaces, and the Damek
 Ricci spaces are harmonic spaces; the classification of harmonic spaces is incomplete as it is not known
 if these are the only possible examples.
 Although our interest is primarily in harmonic spaces, the results of this paper are also applicable in the context
of Riemannian manifolds which are harmonic with respect to a single point; there is relatively less work
in the literature concerning this class of spaces.
And, in particular, the nature of the radial eigen-spaces of the Laplacian on functions and on 1-forms
seems not to have been investigated previously.

\subsection{The Laplacian} Let $\vec x=(x^1,\dots,x^m)$ be an arbitrary system
of local coordinates on a Riemannian manifold $\mathcal{M}=(M,g)$.
Let $\Psi_{\vec x}:=\sqrt{\det(g_{ij})}$ be the volume density function in this
coordinate system; the Riemannian measure is given by
$\operatorname{dvol}=\Psi_{\vec x}dx^1\dots dx^m$ and the scalar Laplacian is given by
$\Delta^0=-\Psi_{\vec x}^{-1}\partial_{x^i}\Psi_{\vec x}g^{ij}\partial_{x^j}$ where we adopt the Einstein convention
and sum over repeated indices.
We shall use the notation $\tilde\Theta_P$
for the volume density function in geodesic coordinates
centered at $P$ and the notation $\Theta_P$ for the volume density function
in geodesic polar coordinates $(r,\theta)$ centered at $P$ for $0<r<\iota_P$.
Let $\operatorname{dvol}_{\vec\theta}$ be the volume element of the sphere of radius
1 in Euclidean space. Since $dx^1\dots dx^m=r^{m-1}dr\operatorname{dvol}_{\vec\theta}$,
$\Theta_P=r^{m-1}\tilde\Theta_P$. In geodesic normal coordinates, $g_{ij}=\delta_{ij}+O(\|x\|^2)$.
Consequently,
\begin{equation}\label{E1.a}
\tilde\Theta_P=1+O(\|x\|^2)\,.
\end{equation}We work on the punctured disk $0<r<\iota_P$ and say that $\phi\in C^\infty(0,\iota_P)$ is a {\it radial harmonic function}
if $\Delta^0\phi^0(r)=0$.
The following is well known (see, for example, the discussion in \cite{B,BTV, K16}).

\begin{theorem}
Let $\mathcal{M}=(M,g)$ be a Riemannian manifold.
\begin{enumerate}
\item The following assertions are equivalent and if any
is satisfied, then $\mathcal{M}$ is said to be a {\bf harmonic space}.
\begin{enumerate}
\item $\Theta_P(r,\theta)=\Theta_P(r)$ is independent of $\theta$ and only depends on $P$ and $r$.
\item $\Theta_P(r,\theta)=\Theta(r)$ is independent of $P$ and $\theta$ and only depends on $r$.
\item There exists a non-constant radial harmonic function near any $P\in M$.
\end{enumerate}
\item If $\mathcal{M}$ is a harmonic space, then $\mathcal{M}$ is Einstein.
\end{enumerate}\end{theorem}

\subsection{Power series expansions}
Expand $\tilde\Theta_P$ in formal power series about a point $P$ of $M$. If $\mathcal{M}$ is
harmonic at the point $P$, then there are no terms of odd degree and we have
\begin{equation}\label{E1.b}
\tilde\Theta_P(r)\sim1+\sum_{n=1}^\infty\mathcal{H}_{2n}r^{2n}\,.
\end{equation}
The asymptotic coefficients $\mathcal{H}_{2n}$ are geometrically determined. If $\xi$ is a unit tangent vector field, let
$\mathcal{J}_0(\xi):=\mathcal{J}(\xi)$ be the Jacobi operator and
let $\mathcal{J}_k(\xi)=\nabla_\xi^k\mathcal{J}(\xi)$.
We refer to \cite{GP} a discussion of the history involved in the proof of the following result.

\begin{theorem}\label{T1.2}
 Let $\mathcal{M}$ be a harmonic space. Let $\xi$ be any unit tangent vector.
 \begin{enumerate}
\item $\mathcal{H}_2= -\displaystyle\frac{1}6\Tr\{\mathcal{J}(\xi)\}$.
\smallbreak\item $\mathcal{H}_4=\displaystyle\frac1{72}\Tr\{\mathcal{J}(\xi)\}^2-\frac{1}{180}\Tr\{\mathcal{J}(\xi)^2\}$.
\smallbreak\item $\mathcal{H}_6=\displaystyle-\frac{\Tr\{\mathcal{J}(\xi)\}^3}{1296}
+\frac{\Tr\{\mathcal{J}\}\Tr\{\mathcal{J}(\xi) ^2\}}{1080}
-\frac{\Tr\{\mathcal{J}(\xi)^3\}}{2835}+\frac{\Tr\{\mathcal{J}_1(\xi)^2\}}{10080}$.
\smallbreak\item $\mathcal{H}_8=\displaystyle\frac{\Tr\{\mathcal{J}(\xi)\}^4}{31104}-\frac{\Tr\{\mathcal{J}(\xi)\}^2 \Tr\{\mathcal{J}(\xi)^2\}}{12960}+\frac{\Tr\{\mathcal{J}(\xi)\} \Tr\{\mathcal{J}(\xi)^3\}}{17010}$
\smallbreak\qquad$\displaystyle-\frac{\Tr\{\mathcal{J}(\xi)\} \Tr\{\mathcal{J}_1(\xi)^2\}}{60480}
+\frac{\Tr\{\mathcal{J}(\xi)^2\}^2}{64800}-\frac{\Tr\{\mathcal{J}(\xi)^4\}}{37800}$
\smallbreak\qquad$\displaystyle-\frac{ \Tr\{\mathcal{J}(\xi)^2\mathcal{J}_2(\xi)\}}{340200}+\frac{\Tr\{\mathcal{J}(\xi)\mathcal{J}_1(\xi)^2\}}{54432}-\frac{\Tr\{\mathcal{J}_2(\xi)^2\}}{907200}$.
\end{enumerate}
\end{theorem}

\subsection{The Laplacian in geodesic polar coordinates}
Let $(r,\theta)$ be godesic polar coordinates centered
at $P\in M$ for $0<r<\iota_P$. Let
$$\Xi_P:=\Theta_P^{-1}\partial_r\Theta_P=\partial_r\log(\Theta_P)\text{ and }
\tilde\Xi_P:=\tilde\Theta_P^{-1}\partial_r\tilde\Theta_P=\partial_r\log(\tilde\Theta_P)\,.
$$
Because $\Theta_P=r^{m-1}\tilde\Theta_P$ and $\tilde\Theta_P=1+O(r^2)$,
$\Xi_P=(m-1)r^{-1}+\tilde\Xi_P$ where $\tilde\Xi_P$
is regular and vanishes to first order at $r=0$.
Since $\partial_r\perp\partial_{\theta^i}$ and $g_{rr}=1$, if $\phi^0=\phi^0(r)$ is a radial function, then
\begin{equation}\label{E1.c}
\textstyle\Delta^0\phi^0=-\{\partial_r^2\phi^0+\Xi_P(r,\theta)\partial_r\phi^0\}
=-\left\{\partial_r^2\phi^0+\left(\frac{m-1}r+\tilde\Xi_P(r,\theta\right)\partial_r\phi^0\right\}\,.
\end{equation}
Thus $\Delta^0$ is singular at $r=0$.
We work with 1-forms rather than with vector fields; duality lets us pass between the two settings.
Let $\Delta^1$ be the Laplace-Beltrami operator on 1-forms. We use the intertwining relation
$d\Delta^0=\Delta^1d$.
Let $\phi^1(r)dr$ be a radial 1-form. Find a primitive $\phi^0$ so $\partial_r\phi^0=\phi^1$ or equivalently
$d\phi^0=\phi^1dr$. We then have an ODE which is singular at $r=0$:
\begin{eqnarray}
&&\Delta^1(\phi^1(r)dr)=\Delta^1(d\phi^0(r))=d\Delta^0\phi^0(r)=-d\{\partial_r^2\phi^0(r)+\Xi_P(r,\theta)\partial_r\phi^0(r)\}\
\nonumber\\[0.05in]
&=&-\partial_r\left\{\partial_r\phi^1(r)+\Xi_P(r,\theta)\phi^1(r)\right\}dr-\phi^1(r)d_\theta\Xi_P(r,\theta)\label{E1.d}\\[0.05in]
&=&\textstyle
-\partial_r\left\{\partial_r\phi^1(r)+\left(\frac{m-1}r+\tilde\Xi_P(r,\theta)\right)\phi^1(r)\right\}dr-\phi^1(r)d_\theta\tilde\Xi_P(r,\theta)\,.
\nonumber\end{eqnarray}

\subsection{Spaces of radial eigen-functions defined by the Laplacian}
Let $\mathfrak{E}_P^0(\lambda)$ (resp. $\mathfrak{E}_P^1(\lambda)$) be the eigen-space of radial functions
(resp. 1-forms):
\begin{eqnarray*}
&&\mathfrak{E}_P^0(\lambda):=\left\{\phi^0\in C^\infty(0,\iota_P):
\Delta^0(\phi^0(r))=\lambda\phi_0(r)\right\},\\
&&\mathfrak{E}_P^1(\lambda):=\left\{\phi^1dr\in C^\infty(0,\iota_P)dr:
\Delta^1(\phi^1(r)dr)=\lambda\phi^1(r)dr\right\}\,.
\end{eqnarray*}
If $\lambda\ne0$, $d$ is an
isomorphism from $\mathfrak{E}_P^0(\lambda)$ to $\mathfrak{E}_P^1(\lambda)$ so it suffices to study
$\mathfrak{E}_P^0(\lambda)$. Let $\mathfrak{H}_P^0$ (resp. $\mathfrak{H}_P^1$) be the space of radial harmonic functions (resp. 1-forms):
$$
\mathfrak{H}_P^0:=\mathfrak{E}_P^0(0)\text{ and }
\mathfrak{H}_P^1:=\mathfrak{E}_P^1(0)\,.
$$
Let $\tilde\Delta^1(\phi^1)=-g^{ij}\phi_{;ij}$
be the \emph{rough Laplacian}. The Bochner formula then yields
$\Delta^1(\phi^1)=\tilde\Delta^1(\phi^1)+\rho(\phi^1)$ where $\rho$ is the Ricci operator.
Let $\widetilde{\mathfrak{H}}_P^1$ be the space of radial 1-forms
in $\ker\tilde\Delta^1$. Consequently, if $\mathcal{M}$ is Einstein, then
$\tilde{\mathfrak{H}}^i_P=\mathfrak{E}^i(\lambda)$
where $\lambda$ is the Einstein constant. We have the following characterizations of what it means for
a space to be harmonic with respect to a point $P$.

\begin{theorem}\label{T1.3} Let $P\in\mathcal{M}$.
\begin{enumerate}
\item The following assertions are equivalent
and if any is satisfied, then one says that $\mathcal{M}$ is harmonic with respect to $P$.
\smallbreak{\rm(a)} $\Theta_P(r,\theta)=\Theta_P(r)$.\hspace{.3cm}{\rm(b)} $\Xi_P(r,\theta)=\Xi_P(r)$.
\hspace{.3cm}{\rm(c)} $\dim\{\mathfrak{H}_P^0\}=2$.
\smallbreak
{\rm(d)} $\dim\{\mathfrak{H}_P^0\}\ge2$.
\hspace{.75cm}{\rm(e)} $\dim\{{\mathfrak{H}_P^1}\}=2$.\hspace{.88cm}
{\rm(f)} $\dim\{{\mathfrak{H}_P^1}\}\ge1$.
\smallbreak\item The following assertions are equivalent if $\lambda\ne0$ and if
any is satisfied, then $\mathcal{M}$ is harmonic with respect to $P$.
\smallbreak{\rm(a)} $\Xi_P(r,\theta)=\Xi_P(r)$.\hfill{\rm(b)} $\dim\{\mathfrak{E}_P^0(\lambda)\}=2$.
\hfill{\rm(c)} $\dim\{\mathfrak{E}_P^0(\lambda)\}\ge1$.\hfill\vphantom{.}
\smallbreak{\rm(d)} $\dim\{\mathfrak{E}_P^1(\lambda)\}=2$.
\hspace{.3cm}{\rm(e)} $\dim\{\mathfrak{E}_P^1(\lambda)\}\ge1$.
\smallbreak\item The following assertions are equivalent if $\mathcal{M}$ is Einstein.
\smallbreak{\rm(a)} $\Xi_P(r,\theta)=\Xi_P(r)$.\hspace{.4cm}{\rm(b)}
$\dim\{\widetilde{\mathfrak H}_P^1\}=2$.
\hspace{.9cm}{\rm(c)} $\dim\{\widetilde{\mathfrak H}_P^1\}\ge1$.
\end{enumerate}\end{theorem}

Suppose that $\mathcal{M}$ is harmonic at the point $P$. Let $\vec x=(x^1,\dots,x^m)$ be {{geodesic}}
coordinates centered at $P$. Then $r=\{(x^1)^2+\dots+(x^m)^2\}^{1/2}$ is the geodesic distance to $P$.
Let $\theta=(1,0,\dots,0)$.
Since $\tilde\Theta_P(r)=\{\det(g_{ij}(\vec x))\}^{\frac12}(r,0,\dots,0)$, $\tilde\Theta_P$ is smooth at $r=0$.
Consequently,
$\tilde\Theta_P$ is an even function of $r$. The expansion of Theorem~\ref{T1.2} generalizes to this setting although there are extra terms. One has, for example,
$$\displaystyle\mathcal{H}_4(\xi)=\frac{\Tr\{\mathcal{J}(\xi)\}^2}{72}
-\frac{\Tr\{\mathcal{J}(\xi)^2\}}{180}-\frac{\Tr\{\mathcal{J}_2(\xi)\}}{40}\,.$$
The extra term $\Tr\{\mathcal{J}_2(\xi)\}$ vanishes if $\mathcal{M}$ is a harmonic space.
Similar formulas for $\mathcal{H}_6$ and $\mathcal{H}_8$ follow from the results in~\cite{GP}.

\subsection{The nature of the singularity at 0}
If $\mathcal{M}=(\mathbb{R}^m,g_e)$ is flat space,
then $\Theta_P(r)=r^{m-1}$ and it is immediate that
\begin{equation}\label{E1.e}
\mathfrak{H}^0=\left\{\begin{array}{ll}\operatorname{Span}\{\text{\bf 1},\log[r]\}&\text{ if }m=2\\
\operatorname{Span}\{\text{\bf 1},r^{2-m}\}&\text{ if }m>2\\
\end{array}\right\}\text{ and }
\mathfrak{H}^1=\operatorname{Span}\{r,r^{1-m}\}{ dr}\,.
\end{equation}
Let $C^\omega\Gronk$ be the ring of real analytic functions on $\Gronk$.
The following result generalizes Equation~(\ref{E1.e}) to the context of real analytic manifolds which
are harmonic with respect to some point.

\begin{theorem}\label{T1.4}
Let $\mathcal{M}$ be real analytic and harmonic with respect to the point $P$.
\begin{enumerate}
\item There is a basis $\{\phi_{0,\lambda}^0,\phi_{1,\lambda}^0\}$ for $\mathfrak{E}_P^0(\lambda)$ so that
\begin{enumerate}
\item $\phi_{0,\lambda}^0\in C^\omega\Gronk$ satisfies
$\displaystyle\phi_{0,\lambda}^0=1-\frac{\lambda}{2m}r^2+\frac{\lambda  (\lambda+4 \mathcal{H}_{2})}{8 m (m+2)} r^4$\newline\vphantom{.}\qquad
$\displaystyle-\frac{\lambda   \left(16 \mathcal{H}_{2}^2 (m+4)+12 \mathcal{H}_{2} \lambda -32 \mathcal{H}_{4} (m+2)+\lambda ^2\right)}
{48 m \left(m^2+6 m+8\right)}r^6+O(r^8)$.
\item There exists $\sigma\in C^\omega\Gronk$ and $C(\mathcal{M})\in\mathbb{R}$
so that
\begin{enumerate}
\item If $m=2$, $\phi_{1,\lambda}^0=\sigma+\phi_{0,\lambda}^0\log(r)$.
\item If $m\ge3$,
$\phi_{1,\lambda}^0=r^{2-m}\sigma+\frac1{m-2}C(\mathcal{M})\phi_{0,\lambda}^0\log(r)$
for $\sigma(0)=1$.
\end{enumerate}\end{enumerate}
\smallbreak\item There is a basis $\{\phi^1_{0,\lambda}dr,\phi^1_{1,\lambda}dr\}$
for $\mathfrak{C}^1_P(\lambda)$
where $\phi^1_{1,\lambda}=\partial_r\phi^0_{1,\lambda}$.
If $\lambda\ne0$, we may take $\phi^1_{0,\lambda}=-\frac m\lambda\partial_r\phi^0_{0,\lambda}$.
If $\lambda=0$, we may take
$\phi^1_{0,0}:=\displaystyle\lim_{\lambda\rightarrow0}\phi^1_{0,\lambda}$. \par\noindent We have
$\displaystyle\phi^1_{0,\lambda}=r-\frac{\lambda+4 \mathcal{H}_{2}}{2(m+2)} r^3$\smallbreak
$\displaystyle\qquad\qquad
+\frac{16 \mathcal{H}_{2}^2 (m+4)+12 \mathcal{H}_{2} \lambda -32 \mathcal{H}_{4} (m+2)+\lambda ^2}
{8(m^2+6 m+8)}r^5+O(r^7)$.
\end{enumerate}\end{theorem}

The constant $C(\mathcal{M})$ controls whether or not the log terms are present in $\phi_{1,\lambda}^0$
and, if $\lambda\ne0$ in $\phi_{1,\lambda}^0$.
A direct computation establishes the following result; we omit the proof since it is a combinatorial exercise.
\begin{theorem} Let $C(\mathcal{M})$ be the constant of Theorem~\ref{T1.4}~(1b). Then
\begin{enumerate}
\item $C(\mathcal{M})$ vanishes if $m$ is odd.
$C(\mathcal{M}^4)=4\mathcal{H}_2-\lambda$.
\smallbreak\noindent
$C(\mathcal{M}^6)=-16\mathcal{H}_2^2+16\mathcal{H}_4+3\mathcal{H}_2\lambda
-\frac{1}4\lambda^2$.
\smallbreak\noindent
$C(\mathcal{M}^8)=-\frac{21 }{4}\mathcal{H}_2^2 \lambda +36 \mathcal{H}_2^3
-72 \mathcal{H}_2 \mathcal{H}_4+\frac{3 }{8}\mathcal{H}_2 \lambda ^2
 +5 \mathcal{H}_4 \lambda +36 \mathcal{H}_6-\frac{1}{64}\lambda ^3$.
\smallbreak\item For the rank-one symmetric spaces in low dimensions, we have
 \begin{enumerate}
\item$C(S^4)=-(\lambda+2)$,\hspace{2.6cm}$C(\mathbb{H}^4)=-(\lambda-2)$,
\item$C(\mathbb{CP}^2)=-(\lambda+4)$,\hspace{2.3cm}$C(\widetilde{\mathbb{CP}^2})=-(\lambda-4)$,
\item$C(S^6)=\frac{-1}4(\lambda+6)(\lambda+4)$,\hspace{1.2cm}
$C(\mathbb{H}^6)=\frac{-1}4(\lambda-6)(\lambda-4)$,
\item$C(\mathbb{CP}^3)=\frac{-1}4(\lambda+8)^2$,\hspace{1.9cm}
$C(\widetilde{\mathbb{CP}^3})=\frac{-1}4(\lambda-8)^2$,
\item$C(S^8)=\frac{-1}{64}(\lambda+6)(\lambda+10)(\lambda+12)$,\hspace{.2cm}
$C(\mathbb{H}^8)=\frac{-1}{64}(\lambda-6)(\lambda-10)(\lambda-12)$,
\item$C(\mathbb{CP}^4)=\frac{-1}{64}(\lambda+12)^2(\lambda+16)$,\hspace{.45cm}
$C(\widetilde{\mathbb{CP}^4})=\frac{-1}{64}(\lambda-12)^2(\lambda-16)$,
\item$C(\mathbb{HP}^2)=\frac{-1}{64}(\lambda+16)(\lambda+24)^2$,\hspace{.45cm}
$C(\widetilde{\mathbb{HP}^2})=\frac{-1}{64}(\lambda-16)(\lambda-24)^2$.
\end{enumerate}\end{enumerate}\end{theorem}

\subsection{Harmonic radial 1-forms} Set $\lambda=0$.
Let $\phi^1$ be a non-trivial harmonic radial 1-form. Then $\mathcal{M}$ is harmonic
at the point $P$. Since $\phi^0_{0,0}=1$, Theorem~\ref{T1.4} shows
there are no log terms in $\phi^1$. If we solve
the equation $\partial_r\phi^1+\Xi_P\phi^1=0$, then $\phi^1$ is harmonic.
Let $\mathcal{M}$ be a geodesically complete harmonic space.
If the scalar curvature $\tau$ vanishes, then $\mathcal{M}$ is flat. If $\tau>0$,
then $M$ is a rank one symmetric space. We have for suitably chosen $(m,k)$ that
\begin{eqnarray*}
&&\Theta=\sin(r)^{m-1}\cos(r)^k,\quad\Xi=(m-1)\cot(r)-k\tan(r),\\
&&\psi(r)dr=\cos(r)^{-k}\sin(r)^{1-m}dr\in\mathfrak{H}^1\,.
\end{eqnarray*}
The only known examples if $\tau<0$ are the Damek Ricci spaces and the negative
curvature rank one symmetric spaces. There we have
\begin{eqnarray*}
&&\Theta=\sinh(r)^{m-1}\cosh(r)^k,\quad\Xi=(m-1)\coth(r)-k\tanh(r),\\
&&\psi(r)dr=\cosh(r)^{-k}\sinh(r)^{1-m}dr\in\mathfrak{H}^1\,.
\end{eqnarray*}
These solutions blow up at the origin;
we were not able to find the corresponding regular solutions in $\mathfrak{H}^1$ in closed form.
 There are no known examples of harmonic spaces which do not fall under this rubric. However, we will
 show presently that there are many manifolds which are harmonic with respect to a point and to
 which Theorem~\ref{T1.4} applies.

\subsection{Constructing manifolds which are harmonic with respect to a point}\label{S1.7}
As noted previously,
Copson and Ruse~\cite{CR40} showed that the conformally flat metric
\begin{equation}\label{E1.f}
ds^2=e^{2\phi(\|x\|^2)}((dx^1)^2+\dots+(dx^m)^2)\text{ for }\phi\in C^\infty
\end{equation}
was harmonic at the origin but, for generic $\phi$, was not harmonic at other points. We now
generalize this formalism. We work locally.

\begin{theorem}\label{T1.6} Let $\mathcal{M}=(M,g)$ be the germ of a Riemannian manifold which is harmonic at the point $P$.
Let $\phi\in C^\infty(\mathbb{R})$ be a smooth even function. Let $g_\phi:=e^{2\phi(r)}g$ be a conformal radial
deformation of $g$. Then $g_\phi$ is harmonic at the point $P$.
\end{theorem}

If we choose as our base manifold $\mathcal{M}$ to be harmonic space which is not locally conformally flat (for example
complex projective space, quaternionic projective space, or the Cayley plane), then $\mathcal{M}$
is not locally conformally flat and hence the metric of Theorem~\ref{T1.6} does not arise from
Equation~(\ref{E1.f}).
The volume density determines the underlying geometry of a harmonic space in many instances.
For example, it is known that a harmonic space with the same volume density as $\mathbb R^n$ is flat.
This is not the case if $\mathcal{M}$ is only assumed to be harmonic at the point $P$.

\begin{theorem}\label{T1.7} If $m\ge4$ is even, there exists a Riemannian manifold $\mathcal{M}$
of dimension $m$
which is harmonic at a point $P$, which has $\Theta_P=r^{m-1}$, which
is diffeomorphic to $\mathbb{R}^m$, and which is not conformally flat.
\end{theorem}

\section{Proof of Theorem~\ref{T1.3}}
\subsection{Linear ordinary differential equations}
We recall the following well known result in the theory of ordinary differential equations.

\begin{theorem}\label{T2.1}
Let $a(r)$ and $b(r)$ be smooth functions on an interval $[a,b]$.
Given any $r_0\in[a,b]$
and given $(u,v)\in\mathbb{R}^2$,
there exists a unique smooth solution to the equation $\partial_r^2\phi(r)+a(r)\partial_r\phi(r)+b(r)\phi(r)=0$
with
$\phi(r_0)=u$ and $\partial_r\phi(r_0)=v$.
If $a$ and $b$ are real analytic, then $\phi$ is real analytic.
\end{theorem}

\subsection{Harmonic radial functions} We use Theorem~\ref{T2.1}
to prove Theorem~\ref{T1.3}~(1).
It is immediate that $\Theta_P$ is radial implies $\Xi_P$ is radial. Conversely, suppose $\Xi_P$ is radial or,
equivalently, $\tilde\Xi_P$ is radial. Then $\tilde\Theta_P(r,\theta)$ satisfies the ODE
$$
\partial_r\tilde\Theta_P(r,\theta)-\tilde\Xi_P(r)\tilde\Theta_P(r,\theta)=0\text{ with }\tilde\Theta_P(0,\theta)=1\,.
$$
Since the coefficients of the ODE and
the initial condition are independent of $\theta$, $\tilde\Theta_P(r,\theta)$ is independent of $\theta$.
Consequently, Assertions (1a) and (1b) are equivalent.
Suppose that Assertion~(1b) holds so $\Xi_P(r,\theta)=\Xi_P(r)$ is independent of $\theta$ and Equation~(\ref{E1.c})
becomes an ODE.
By Theorem~\ref{T2.1}, the space of solutions to a second order ODE is 2-dimensional. Thus Assertion~(1b)
implies Assertion~(1c). It is immediate that Assertion~(1c) implies Assertion~(1d).
Suppose that Assertion~(1d) holds.
The constant function $\text{\bf1}$ belongs to $\mathfrak{H}_P^0$. Suppose $\dim\{\mathfrak{H}_P^0\}\ge2$.
Then there exists a non-constant harmonic radial function $\phi_P^0$.
Fix $\theta$. Then $\phi_P^0$ is determined by $\phi_P^0(r_0)$ and $\partial_r\phi_P^0(r_0)$ for any $r_0$.
Suppose there exists $r_0$ so $\partial_r\phi_P^0(r_0)=0$. Set $\psi_P^0(r)=\phi_P^0(r)-\phi_P ^0(r_0)\text{\bf1}$; this is harmonic.
Since $\psi_P^0(r_0)=0$ and $\partial_r\psi_P^0(r_0)=0$, $\psi_P^0$ vanishes identically.
This implies $\phi_P^0$ is constant which is false.
Thus $\partial_r\phi_P^0$ never vanishes and $\Xi_P(r,\theta)=-\{\partial_r\phi_P^0(r)\}^{-1}\partial_r^2\phi_P^0(r)$
is independent of $\theta$. This shows that Assertion~(1d) implies Assertion~(1b).

Suppose that Assertion~(1b) holds. Then Equation~(\ref{E1.d}) becomes a second order ODE and,
by Theorem~\ref{T2.1}, the solution space has dimension 2. This shows that Assertion~(1b) implies Assertion~(1e).
Clearly Assertion~(1e) implies Assertion~(1f). Suppose that Assertion~(1f) holds. Let $\phi^1$ be a non-trivial solution to Equation~(\ref{E1.d}).
Fix $\theta_0$ and $r_0$.
Suppose there exists a sequence $r_n\rightarrow r_0$ so $\phi^1(r_n)=0$. By the mean value property, there
exist points $s_n$ between $r_n$ and $r_{n+1}$ so $\partial_r\phi^1(s_n)=0$. This implies that
$\phi^1(r_0)=\partial_r\phi^1(r_0)=0$.
Because the function $\phi^1$ solves the ODE
$\partial_r^2\phi^1(r)+\Xi_P(r,\theta_0)\partial_r\phi^1(r)+(\partial_r\Xi_P(r,\theta_0))\phi^1(r)=0$,
we may use Theorem~\ref{T2.1} to see $\phi^1$
vanishes identically which is false. Thus we can find a sequence $r_n\rightarrow r_0$ so
$\phi^1(r_n)\ne0$. The equation $\phi^1(r_n)d_\theta\tilde\Xi_P(r_n,\theta_0)=0$
implies $d_\theta\Xi_P(r_n,\theta)=0$ and hence by continuity $d_\theta\Xi_P(r_0,\theta)=0$.
Since $\theta_0$ was arbitrary, this shows that
$\Xi_P(r_0,\theta)$
is independent of $\theta$ so Assertion~(1b) holds.~\qed

\subsection{The proof of Theorem~\ref{T1.3}~(2)}
If $\phi$ is a radial eigen-function, then Equation~(\ref{E1.c}) yields, after
taking into account the sign convention for $\Delta^0$,
\begin{equation}\label{E2.a}
\partial_r^2\phi^0(r)+\left(\frac{m-1}r+\tilde\Xi_P(r,\theta)\right)\partial_r\phi^0(r)+\lambda\phi^0(r)=0\,.
\end{equation}
Suppose $\Xi_p(r,\theta)=\Xi_p(r)$. Then Equation~(\ref{E2.a}) is a second order ODE and Theorem~\ref{T2.1}
yields $\dim\{\mathfrak{E}^0_P(\lambda)\}=2$. Thus Assertion~(2a) implies Assertion~(2b). It
is immediate that Assertion~(2b) implies Assertion~(2c). Suppose Assertion~(2c) holds. Let
$r_0\in(0,\iota_P)$ and let $\phi^0$ be a non-trivial solution to Equation~(\ref{E2.a}). Suppose
$\partial_r\phi^0(r_0)\ne0$. Subtracting Equation~(\ref{E2.a}) at two values of $\theta$ shows
$(\tilde\Xi_P(r_0,\theta_1)-\tilde\Xi_P(r_0,\theta_0))\partial_r\phi^0(r_0)=0$ so
$\tilde\Xi_P(r_0,\theta_1)(r_0)=\tilde\Xi(r_0,\theta_0)(r_0)$. Next, suppose $\partial_r\phi^0(r_0)=0$ but
there exist a sequence of points $r_n\rightarrow r_0$ so $\partial_r\phi^0(r_n)\ne0$.
Since $\tilde\Xi_P(r_0,\theta_1)(r_n)=\tilde\Xi(r_0,\theta_0)(r_n)$,
$\tilde\Xi_P(r_0,\theta_1)(r_0)=\tilde\Xi(r_0,\theta_0)(r_0)$ by continuity. Suppose
$\partial_r\phi^0$ vanishes identically near $r_0$. We then have $\partial_r^2\phi^0$ vanishes
identically near $r_0$ and hence, since $\lambda\ne0$, $\phi^0$ vanishes identically near $r_0$.
Since $\partial_r\phi^0(r_0)=\phi^0(r^0)=0$, Theorem~\ref{T2.1} shows $\phi^0$ vanishes identically
on $(0,\iota_P)$ which is false. This contradiction shows Assertion~(2c) implies Assertion~(2a).
Since $\lambda\ne0$, $d:\mathfrak{E}^0_P(\lambda)\rightarrow\mathfrak{E}^1_P(\lambda)$
is an isomorphism so Assertion~(2b) is equivalent to Assertion~(2d) and Assertion~(2c) is equivalent
to Assertion~(2e).~\qed

\subsection{Proof of Theorem~\ref{T1.3}~(3)}
Since $\mathcal{M}$ is Einstein, $\tilde{\mathfrak{H}}^1_P=\mathfrak{E}_P^1(\lambda)$ where
$\lambda$ is the Einstein constant. If $\lambda=0$, then Theorem~\ref{T1.3}~(3) follows from
Theorem~\ref{T1.3}~(1). If $\lambda\ne0$, then Theorem~\ref{T1.3}~(3) follows from
Theorem~\ref{T1.3}~(2).~\qed

\section{The proof of Theorem~\ref{T1.4}}
We begin by examining the eigenvalue equation for $\Delta^0$. Set
\begin{equation}\label{E3.a}
\mathcal{P}\psi:=\partial_r^2\psi+\left(\frac{m-1}r+\tilde\Xi_P(r)\right)\partial_r\psi+\lambda\psi\,.
\end{equation}

\begin{lemma}\label{L3.1}
Assume that $\tilde\Xi_P\in C^\omega\Gronk$.
Let $\eta\in C^\omega\Gronk$. Then there exists $\psi\in C^\omega\Gronk$ with
$\psi(0)=1$, $\partial_r\psi(0)=0$, and $\mathcal{P}\psi=\eta$.
\end{lemma}

\begin{proof}
We can use power series to solve Equation~(\ref{E3.a}) away from $r=0$ where the equation
is regular and use Theorem~\ref{T1.2} to patch the solutions together. Consequently,
the crux of the matter is what happens near $r=0$. Since $\tilde\Theta_P=1+O(r^2)$,
$\tilde\Xi_P=O(r)$. We expand $\tilde\Xi_P$, $\eta$, and $\psi$ power series:
\begin{eqnarray*}
\Xi_P(r)=\sum_{k=1}^\infty\xi_kr^k,\qquad\eta(r)=
\sum_{k=0}^\infty\eta_kr^k,\quad\psi(r)=1+\sum_{k=2}^\infty\psi_kr^k\,.
\end{eqnarray*}
We proceed purely formally and expand
\medbreak\quad$\displaystyle\partial_r^2\psi+\frac{m-1}r\partial_r\psi=
\sum_{k=2}^\infty k(k+m-2)\psi_k r^{k-2}$,
\medbreak\quad$r\xi\partial_r\psi=\hspace{1.75cm}
\displaystyle\displaystyle\sum_{k=2}^\infty\ \sum_{i+j=k-2}j\xi_i\psi_jr^{k-2}$,
\medbreak\quad$\lambda\psi-\eta=\hspace{1.65cm}\displaystyle\sum_{k=2}^\infty\{\lambda\psi_{k-2}-\eta_{k-2}\}r^{k-2}$,
\medbreak\quad$\mathcal{P}(\psi)-\eta=
\displaystyle\sum_{k=2}^\infty\bigg\{k(m+k-2)\psi_k+\lambda\psi_{k-2}-\eta_{k-2}
+\sum_{i+j=k-2}j\xi_i\psi_j\bigg\}r^{k-2}$.
\medbreak\noindent In the sum $i+j=k-2$, we have $i\ge1$ and thus $j\le k-3$ so $\psi_j$ has been determined
at a previous stage; this sum is empty, of course, if $k=2$ or $k=3$.
To ensure $\mathcal{P}\psi=\eta$, we obtain the recursive relations for $k\ge2$
$$\displaystyle\psi_k=\frac{-1}{k(m+k-2)}\bigg\{\lambda\psi_{k-2}-\eta_{k-2}
+\sum_{i+j=k-2}j\xi_i\psi_j\bigg\}\,.
$$
\medbreak\noindent We must estimate $\psi_k$ to establish convergence of the series defining $\psi$. Since
$\tilde\Xi_P$ and $\xi$ are real analytic,
there exists $\kappa>1$ so
$|\xi_\nu|\le\kappa^{\nu+1}$ and $|\eta_\nu|\le\kappa^{\nu+1}$ for all $0\le\nu$.
By increasing $\kappa$ if necessary, we can also assume that $|\lambda|\le\kappa$.
We assume inductively that $|\psi_j|\le\kappa^{j+1}$ for $j\le k-1$.
We may then estimate:
\begin{eqnarray*}
|\psi_k|&\le&\frac1{k(m+k-2)}\bigg\{\kappa^k+\kappa^{k-1}+\sum_{i+j=k-2}j\,\kappa^{i+1+j+1}\bigg\}\\
&\le&\frac{\kappa^k(2+(k-1)^2)}{k^2}\le\kappa^{k+1}\,.
\end{eqnarray*}
This ensures the series defining $\psi$ converges near $r=0$.
\end{proof}

\subsection{The proof of Theorem~\ref{T1.4}~(1a)} Use Lemma~\ref{L3.1} to find $\phi_{0,\lambda}^0$
so $\phi_{0,\lambda}^0(0)=1$ and so $\mathcal{P}(\phi_{0,\lambda}^0)=0$. The expansion of
$\phi_{0,\lambda}^0$ about the origin is then a routine computation which is omitted.~\qed

\subsection{The proof of Theorem~\ref{T1.4}~(1b)}
Because $\mathcal{P}(\phi_{0,\lambda}^0)=0$, we have
\begin{eqnarray*}
\mathcal{P}(\phi_{0,\lambda}^0(r)\log(r))
&=&\mathcal{P}(\phi_{0,\lambda}^0)\cdot\log(r)+\phi_{0,\lambda}^0(r)\cdot\partial_r^2\log(r)\\
&&+2\partial_r\phi_{0,\lambda}^0(r)\cdot\partial_r\log(r)\\
&&+\phi_{0,\lambda}^0(r)((m-1)r^{-1}+\tilde\Xi_p(r))\partial_r\log(r)\,.
\end{eqnarray*}
Since $\tilde\Xi_P=O(r)$, $\phi_{0,\lambda}^0=1+O(r^2)$, and $\mathcal{P}(\phi_{0,\lambda}^0)=0$,
we may express
\begin{equation}\label{E3.b}
\mathcal{P}(\phi_{0,\lambda}^0(r)\log(r))=(m-2)r^{-2}+\eta(r)\text{ for }\eta\in C^\omega\Gronk\,.
\end{equation}

\subsubsection{Suppose $m=2$}  Apply Lemma~\ref{L3.1} to
choose $\sigma\in C^\omega\Gronk$ so $\mathcal{P}(\sigma)=-\eta$.
Lemma~\ref{T1.4}~(1b-i) follows since $\phi_{1,\lambda}^0:=\phi_{0,\lambda}^0\log(r)+\sigma$
satisfies $\mathcal{P}(\phi_{1,\lambda}^0)=0$.~\qed

\subsubsection{Suppose $m=3$}
We compute that
$\eta:=\textstyle\mathcal{P}(r^{-1}+\frac12(\xi_1-\lambda)r)\in C^\omega\Gronk$.
Use Lemma~\ref{L3.1} to find
$\sigma$ so $\mathcal{P}(\sigma)=-\eta$. Let
$\phi_{1,\lambda}^0:=r^{-1}+\frac12(\xi_1-\lambda)r+\sigma(r)$ and set $C(\mathcal{M})=0$. We then have
$\mathcal{P}(\phi_{1,\lambda}^0)=0$.~\qed

\subsubsection{Suppose $m=4$}
We compute that
$\textstyle\mathcal{P}(r^{-2}+\frac23\xi_2r)=r^{-2}(\lambda-2\xi_1)+\eta_1$ where $\eta\in C^\omega\Gronk$.
Consequently, by Equation~(\ref{E3.b}),
$$\textstyle
\mathcal{P}\big(r^{-2}+\frac1{m-2}(2\xi_1-\lambda)\phi_{0,\lambda}^0(r)\log(r)+\frac23\xi_2r\big)=\eta\in C^\omega\Gronk\,.
$$
Use Lemma~\ref{L3.1} to find
$\sigma_1$ so $\mathcal{P}(\sigma_1)=-\eta$. Set
$C(\mathcal{M}):=2\xi_1-\lambda=4\mathcal{H}_2-\lambda$ and
\smallbreak\hfill$\phi^0_{1,\lambda}:=r^{-2}+\frac1{m-2}(2\xi_1-\lambda)\phi_{0,\lambda}^0(r)\log(r)
+\frac23\xi_2r +\sigma_1$.\hfill\vphantom{.}~\qed

\subsubsection{Suppose $m\ge5$}
Let $p(r)=p_0+p_1r+\dots+p^{m-3}\text{ where }p_0=1\text{ and }p_1=0$.
We may then express
\begin{eqnarray*}
\mathcal{P}(r^{2-m}p(r))&=&\sum_{k=0}^{m-3}(2-m+k)kp_kr^{k-m}\\
&&+\sum_{i,j}(2-m+j)\xi_ip_jr^{i+j+1-m}+\lambda\sum_\ell p_\ell r^{2+\ell-m}\,.
\end{eqnarray*}
The first sum begins with $k=2$ since the term with $p_0$ does not appear and since we took $p_1=0$.
We re-index the second and third sums to express
\medbreak\quad$\displaystyle\mathcal{P}(r^{2-m}p(r))=\sum_{k=2}^{m-3}r^{k-m}
\{(2-m+k)kp_k+\lambda p_{k-2}+\sum_{i+j=k-1}(2-m+j)\xi_ip_j\}$.
\medbreak\noindent Since $i\ge 1$, $j\le k-2$ in the sum $i+j=k-1$.
For $2\le k\le m-3$, we may choose $p_k$ recursively to ensure the
coefficient of $r^{k-m}$ vanishes. We then have
$$
\mathcal{P}(r^{2-m}p(r))=C_1\cdot r^{-2}+C_2r^{-1}+O(1)\,.
$$
We use an appropriate multiple of $\psi(r)\log(r)$ to cancel the $C_1$ term
and an appropriate multiple of $r^{-1}$ to cancel the $C_2$ term.
The remainder of the argument is as given above.~\qed

\subsection{The proof of Theorem~\ref{T1.4}~(2)} We have $d:\mathfrak{E}_P^0(\lambda)\rightarrow\mathfrak{E}_P^1(\lambda)$. In particular,
since $\partial_r\phi_{1,\lambda}^0\ne0$, $\phi_{1,\lambda}^1:=\partial_r\phi_{1,\lambda}^0dr$
can be chosen as an element of a basis for $\mathfrak{E}_P^0(\lambda)$. If $\lambda\ne0$,
$d:\mathfrak{E}_P^0(\lambda)\rightarrow\mathfrak{E}_P^1(\lambda)$ is an isomorphism so
$\partial_r\phi_{0,\lambda}^0dr$ could be chosen as a basis element.
The construction above shows that $\phi_{0,\lambda}^0(r)$ is analytic in $(r,\lambda)$. We expand
$\phi_{0,0}^0=a_0\cdot 1+a_1\cdot\phi_{0,0}^1$. Since $\phi_{0,0}^1$ is singular at $r=0$,
we conclude $a_1=0$. Since $\phi_{0,0}^0(0)=1$, we conclude $\phi_{0,0}^0=1$. Consequently
$\partial_r\phi_{0,\lambda}^0|_{\lambda=0}=0$.  This implies that $\partial_r\phi_{0,\lambda}^0$
is divisible by $\lambda$ and hence we can normalize the remaining basis element to be
$-\frac m\lambda\partial_r\phi_{0,\lambda}^0$. The expansion given for $\phi_{0,\lambda}^1$
then follows from Assertion~1 and, in particular, $\phi_{0,\lambda}^1$ is non-trivial.
Since $\phi_{0,\lambda}^1$ is analytic in $\lambda$ and well
defined for $\lambda=0$, we conclude
by continuity that $\phi_{0,0}^1dr$ is harmonic.~\qed

\section{The proof of Theorem~\ref{T1.6}}
$\mathcal{M}=(M,g)$ be harmonic at the point $P$.
Let $(r,\theta)$ be geodesic polar coordinates centered at a point $P$. Choose local coordinates
$\theta=(\theta^1,\dots,\theta^{m-1})$ on
the unit sphere to express $g=dr^2+g_{ab}(r,\theta)d\theta^ad\theta^b$
and $\Theta_P(r,\theta)=\det(g_{ab}(r,\theta))^{\frac12}$. We have
$$
g_\phi=e^{2\phi(r)}g=e^{2\phi(r)}dr^2+e^{2\phi(r)}g_{ab}(r,\theta)d\theta^ad\theta^b\,.
$$
Let $\rr(r)$ satisfy $\rr(0)=0$ and $d\rr=e^{\phi(r)}dr$. The inverse function theorem then
permits us to express $r=r(\rr)$ and
$g_\phi=d\rr^2+e^{2\phi(r(\rr))}g_{ab}(r(\rr),\theta)d\theta^ad\theta^b$.
Consequently, $(\rr,\theta)\rightarrow r(\rr)\cdot\theta$ gives geodesic polar coordinates for the metric $g_\phi$ and
$\rr$ is the geodsic distance function for $g_\phi$.
Since
\begin{equation}\label{E4.a}
\Theta_{P,g_\phi}(\rr,\theta)=e^{(m-1)\phi(r(\rr))}\Theta_{P,g}(r(\rr))\,,
\end{equation}
we conclude that $g_\phi$ is harmonic at the point $P$ as well.~\qed

\section{The proof of Theorem~\ref{T1.7}} Let $m=2\mathfrak{m}\ge4$.
Then $\mathbb{CP}^{\mathfrak{m}}$ is not conformally flat.
Let $\mathcal{M}:=(\mathbb{CP}^{\mathfrak{m}}-\mathbb{CP}^{\mathfrak{m}-1},g)$ where $g$
is the Fubini-Study metric. We have removed the cut-locus
and consequently, the underlying manifold is an open geodesic ball of radius $\frac\pi2$.
Choose $\phi$ so $e^{(m-1)\phi(r)}\tilde\Theta_{P,g}(r)=1$.
Then Equation~(\ref{E4.a}) ensures $\tilde\Theta_{P,g_\phi}=1$.~\qed
\section*{Acknowledgement} This work was supported by the National Research Foundation of Korea (NRF) grant
funded by the Korea government (MSIT) (NRF-2019R1A2C1083957) and by Project MTM2016-75897-P (AEI/FEDER, UE).
The second author is grateful to Korea Institute for Advanced Study for {their} hospitality.

\end{document}